\documentclass[11pt]{article}
\usepackage{amsmath, amsthm, amsfonts, amssymb}

\def\fullpage {
\addtolength{\topmargin}{-2 cm}
\addtolength{\oddsidemargin}{-0.9cm} \addtolength{\textwidth}{+2 cm}
\addtolength{\textheight}{+4 cm}}
\fullpage

\newtheorem{thm}{Theorem}
\newtheorem*{thmall}{Asymmetric Local Lemma}
\newtheorem{lemma}[thm]{Lemma}

\newtheorem{prop}[thm]{Proposition}

\theoremstyle{remark}

\theoremstyle{definition}
\newtheorem{definition}[thm]{Definition}
\newcommand{\Ex}{\mathop{\bf E\/}}

\title{Coloring sparse hypergraphs}

\begin{document}

\author{
\quad{Jeff Cooper}
\thanks{Department of Mathematics, Statistics, and Computer
Science, University of Illinois at Chicago, Chicago, IL 60607, USA;  email:
jcoope8@uic.edu}
\quad{Dhruv Mubayi}
\thanks{Department of Mathematics, Statistics, and Computer Science, University of Illinois at Chicago, Chicago IL 60607, USA; research supported in part by NSF grants DMS-0969092 and DMS-1300138; email:  mubayi@uic.edu}
}

\maketitle

\begin{abstract}
Fix $k \geq 3$, and let $G$ be a $k$-uniform hypergraph with maximum degree $\Delta$.
Suppose that for each $l = 2, \dots, k-1$, every set of $l$ vertices of $G$
is in at most $\Delta^{\frac{k-l}{k-1}}/f$ edges. Then the chromatic number of
$G$ is $O((\frac{\Delta}{\log f})^{\frac{1}{k-1}})$.
This extends results of Frieze and the second author \cite{fmcoloringk} and Bennett and Bohman \cite{benboh}.

A similar result is proved for $3$-uniform hypergraphs where every vertex lies in few triangles.
This generalizes a result of Alon, Krivelevich, and Sudakov \cite{aloncoloringsparse}, who
proved the result for graphs.

Our main new technical contribution is a deviation inequality for
positive random variables with expectation less than $1$. 
This may be of independent interest and have further applications.
\end{abstract}

\section{Introduction}
A hypergraph $G$ is a tuple consisting of a set of vertices $V$
and a set of edges $E$, which are subsets of $V$;
we will often associate $G$ with its edge set $E$.
A hypergraph has rank $k$ if
every edge contains between $2$ and $k$ vertices and is $k$-uniform if
every edge contains exactly $k$ vertices. A proper coloring of $G$ is
an assignment of colors to the vertices so that no edge is monochromatic.
The chromatic number, $\chi(G)$, is the minimum number of colors
in a proper coloring of $G$.

A hypergraph is \emph{linear} if every pair of vertices is contained in at most one edge.
A \emph{triangle} in a hypergraph $G$ is a set of three distinct edges
$e,f,g \in G$ and three distinct vertices $u,v,w \in V(G)$ such that
$u,v \in e$, $v, w \in f$, $w, u \in g$, and $\{u,v,w\} \cap e \cap f \cap g = \emptyset$.
For example, the three triangles in a $3$-uniform hypergraph are 
$C_3 = \{abc, cde, efa\}$, $F_5 = \{abc, abd, ced\}$, and $K_4^- = \{abc, bcd, abd\}$.

The \emph{degree} of a vertex $u \in V(G)$ is the number of edges containing that vertex.
The \emph{maximum degree} of a hypergraph $G$ is the maximum degree of a vertex $v \in V(G)$.
Improving on results of Catlin~\cite{catlin}, Lawrence~\cite{lawrence}, 
Borodin and Kostochka~\cite{borodinkostochka}, and Kim~\cite{kimc4},
Johansson \cite{johansson} showed that if $G$ is a triangle-free graph
with maximum degree $\Delta$,
then 
\begin{equation}\label{johansson}
\chi(G) = O(\frac{\Delta}{\log \Delta}).
\end{equation}
Random graphs show that the $\log\Delta$ factor in \eqref{johansson} is optimal.
Recently, Frieze and the second author~\cite{fmcoloringk} generalized \eqref{johansson}
to linear $k$-uniform hypergraphs, and the current authors~\cite{cmcoloring3} proved
slightly stronger results for $k=3$. 
\begin{thm}[\cite{fmcoloringk}]\label{lineark}
Fix $k \geq 3$.
If $G$ is a $k$-uniform, linear hypergraph with maximum degree $\Delta$, then
\[
\chi(G) = O (\frac{\Delta}{\log \Delta})^{\frac{1}{k-1}}.
\]
\end{thm}

\begin{thm}[\cite{cmcoloring3}]\label{trifree3}
If $G$ is a $3$-uniform, triangle-free hypergraph with maximum degree $\Delta$, then
\[
\chi(G) = O( (\frac{\Delta}{\log \Delta})^{1/2} ).
\]
\end{thm}

Alon, Krivelevich, and Sudakov~\cite{aloncoloringsparse}
extended \eqref{johansson} by showing that if every vertex $u \in V(G)$ is in at most $\Delta^2/f$
triangles, then $\chi(G) = O(\frac{\Delta}{\log f})$, where $\Delta \to \infty$.
They used this to show that if $G$ contains no copy of $H$, where $H$ is a fixed graph such that 
$H-u$ is biparite for some $u \in V(H)$, then $\chi(G) = O(\Delta/\log\Delta)$.
In this paper, we give similar improvements to the results of \cite{fmcoloringk} and \cite{cmcoloring3}.

Given two hypergraphs $H$ and $A$, 
a map $\phi:V(H) \to V(A)$ is an isomorphism if for all $E \subset V(H)$,
$\phi(E) \in A$ if and only if $E \in H$.
If there exists an isomorphism $\phi:V(H) \to V(A)$, 
we say $H$ is isomorphic to $A$ and denote this by $H \cong_{\phi} A$.
Given a hypergraph $H$ and $v \in V(H)$, let
\[
\Delta_{H,v}(G) = \max_{u \in V(G)} 
| \{ A \subset G: A \cong_{\phi} H \text{ and } \phi(u)=v \}|
\]
and
\[
\Delta_H(G) = \min_{v \in V(H)} \Delta_{H,v}(G). 
\]
For example, suppose $G$ is a $d$-regular graph, and $H$ is the path with edges $xy$ and $yz$.
Then $\Delta_{H,x}(G) = \Delta_{H,z}(G) = d(d-1),$ while $\Delta_{H,y}(G) = \binom{d}{2}.$
Thus $\Delta_H(G) = \binom{d}{2}$.
Another example appears after the statement of Theorem \ref{corlin}.

We improve the result of \cite{cmcoloring3} with the following theorem.
For a hypergraph $H$, let $v(H) = |V(H)|$.
\begin{thm}\label{cortri}
Let $G$ be a $3$-uniform hypergraph with maximum degree $\Delta$.
Let $\mathcal{T}$ denote the set of $3$-uniform triangles.
If 
\[
\Delta_H(G) \leq \Delta^{(v(H)-1)/2}/f
\]
for all $H \in \mathcal{T}$, then
\[
\chi(G) = O((\frac{\Delta}{\log f})^{1/2}).
\]
\end{thm}
\noindent
Notice that the hypotheses of Theorem \ref{cortri} are satisfied when $G$ is linear and $f = \Delta^{1/2}$,
so Theorem \ref{cortri} implies Theorem \ref{lineark} for $3$-uniform hypergraphs.

Given a rank $k$ hypergraph $G$ and $A \subset V(G)$, let 
\[
d_j(A) = |\{B \in G: |B|=j, A \subset B\}|
\]
for each $1 \leq j \leq k$. When the hypergraph is $k$-uniform, we write $d(A)$ instead of $d_k(A)$.
For $1 \leq l \leq j \leq k$, define the maximum $(j,l)$-degree of $G$, denoted $\Delta_{j,l}(G)$,
to be $\max_{A \subset V(G): |A| = l}d_j(A)$ and the maximum $j$-degree of $G$ to be $\Delta_{j,1}(G)$.

Consider the random greedy algorithm for forming an independent set $I$ in a hypergraph $G$:
at each step of the algorithm, a vertex $v$ is chosen at random from the set of vertices $V(G)-I$
such that $I \cup \{v\}$ contains no edge of $G$.
The algorithm terminates when no such $v$ exists.
Notice that when the vertex set of $G$ consists of the edges of a complete graph,
and the edges of $G$ correspond to triangles in this graph,
the random greedy independent set algorithm reduces to the triangle-free process 
(see \cite{trifreebohman}, \cite{bohmankeevash13}, \cite{trifreelong13}).
Bennett and Bohman~\cite{benboh} have recently shown that this algorithm terminates with a large independent
set $I$ with high probability, unifying many of the previous results on $H$-free processes.
Define the $b$\emph{-codegree} of a pair of distinct vertices $v,v'$
to be the number of edges $e,e' \in G$ such that $v \in e$, $v' \in e'$, and $|e \cap e'| = b$.
Let $\Gamma_b(G)$ be the maximum $b$-codegree of $G$.
\begin{thm}[Bennett, Bohman \cite{benboh}]
Let $k \geq 2$ and $\epsilon > 0$ be fixed. 
Let $G$ be a $k$-uniform, $D$-regular hypergraph on $N$ vertices such that $D > N^{\epsilon}$.
If
\[
\Delta_{k,l}(G) < D^{\frac{k-l}{k-1}-\epsilon}  \hspace{3pt}\text{ for } l=2,\dots,k-1
\]
and $\Gamma_{k-1}(G) < D^{1-\epsilon}$ then the random greedy independent set algorithm
produces an independent set $I$ in $G$ with 
\[
|I| = \Omega(N ( \frac{\log N}{D} ) ^{\frac{1}{k-1}} )
\]
with probability $1- \exp\{-N^{\Omega(1)}\}$.
\end{thm}

Our next theorem improves and extends Bennett and Bohman's result on the independence number to chromatic number
when $k \geq 3$.
We also weaken the hypothesis by not requiring any condition on $\Gamma_{k-1}(G)$.
Note that an important aspect of \cite{benboh} is that the random greedy procedure
results in a large independent set. We do not make any such claims in our result below.
\begin{thm}\label{corlin}
Fix $k \geq 3$.
Let $G$ be a $k$-uniform hypergraph with maximum degree $\Delta$.
If
\[
\Delta_{k,l}(G) \leq \Delta^{\frac{k-l}{k-1}}/f \hspace{3pt}\text{ for } l = 2,\dots, k-1,
\]
then
\[
\chi(G) = O( (\frac{\Delta}{\log f})^{\frac{1}{k-1}}).
\]
\end{thm}
The following example shows that Theorem \ref{cortri} implies Theorem \ref{corlin}
when $k=3$
and also motivates our definition of $\Delta_H$:
Suppose $G$ is a $3$-uniform hypergraph with maximum $3$-degree $\Delta$
and maximum $(3,2)$-degree at most $\Delta^{1/2}/f$.
Recall that $F_5 = \{abc, abd, ced\}$.
Notice $\Delta_{F_5, e} \leq \Delta^2$, while $\Delta_{F_5, a} \leq \Delta^2/f^2$;
thus $\Delta_{F_5} \leq \Delta^2/f^2$. 
We also have $\Delta_{C_3} \leq \Delta^{5/2}/f$ and $\Delta_{K_4^-} \leq 4\Delta^{3/2}/f$.
Theorem \ref{cortri} therefore implies $\chi(G) = O( (\Delta/\log f)^{1/2} )$.

Theorem \ref{cortri} and Theorem \ref{corlin} both follow from a general partitioning lemma, which is
a generalization of the main result of \cite{aloncoloringsparse} to hypergraphs.
\begin{definition}
Let $G$ be a rank $k$ hypergraph.
$G$ is $(\Delta, \omega_2,\dots,\omega_k)$\emph{-sparse} if 
$G$ has maximum $k$-degree at most $\Delta$, and for all
$1 \leq l < j \leq k$, $G$ has maximum $(j,l)$-degree at most $\Delta^{\frac{j-l}{k-1}}\omega_j$.
\end{definition}

\noindent
Recall that a hypergraph $H$ is \emph{connected} if for all $u,v \in V(H)$,
there exists a sequence of edges $e_1, \dots, e_n \in H$ such that $u \in e_1$,
$v \in e_n$, and $e_i \cap e_{i+1} \neq \emptyset$ for $1 \leq i < n$.

\begin{lemma}\label{partitionlemma}
Fix $k \geq 2$.
Let $G$ be a rank $k$ hypergraph, and 
let $\mathcal{H}$ be a finite family of fixed, connected hypergraphs.
Let $f = \Delta^{O(1)}$, where $f$ is sufficiently large.
Suppose that
\begin{itemize}
\item $G$ is $(\Delta, \omega_2, \dots, \omega_k)$-sparse, where $\omega_j = \omega_j(\Delta) = f^{o(1)}$
for all $2 \leq j \leq k$.
\item For all $H \in \mathcal{H}$, $\Delta_H(G) \leq \Delta^{\frac{v(H)-1}{k-1}}/f^{v(H)}$.
\end{itemize}
Then $V(G)$ can be partitioned into $O( \Delta^{\frac{1}{k-1}}/f )$ parts
such that the hypergraph induced by each part is $\mathcal{H}$-free
and has maximum $j$-degree at most $2^{2k}f^{j-1}\omega_j$,
for each $1 \leq j \leq k$.
\end{lemma}

Our main new contribution is a technical lemma (Lemma \ref{hittinglemma} in Section \ref{devinequality})
that allows us to prove a large deviation inequality for positive random variables
with mean less than $1$. This may be of independent interest.

\subsection{Organization}
In Section 2, we use Lemma \ref{partitionlemma} to prove Theorems \ref{cortri} and \ref{corlin}.
The remaining sections are devoted to Lemma \ref{partitionlemma}. 
In Section 3, we state two concentration theorems that we will use in Section 5.
In Section 4, we prove our main probabilistic lemma, which is also used in Section 5. 
Section 5 contains the proof of Lemma \ref{partitionlemma}.

\subsection{Notation}\label{notation}
All of our $O$ and $o$ notation is with respect to the parameter $\Delta \to \infty$.
In particular, when we write $f = O(g)$, we mean that there is a positive constant $c$
such that $f < cg$.

Given a hypergraph $G$ and $A \subset V(G)$, we define the $(j-|A|)$-uniform hypergraph
\[
L_{G, j}(A) = \{ B-A: B \in G \text{ with } A \subset B \text{ and } |B| = j\}.
\]
Let $N_{G,j}(A) = V(L_{G,j}(A))$.
When $G$ is $k$-uniform, we drop the subscript $j$. When $G$ is clear from the context,
we drop the subscript $G$.

\section{Proof of Theorem \ref{cortri} and \ref{corlin}}
We first prove a simple bound on the chromatic number of general rank $k$ hypergraphs.
The proof is based on a proof of Erd\H{o}s and Lov\'asz~\cite{erdoslovasz} for $k$-uniform hypergraphs.
This bound will allow us to assume that $f$ is sufficiently large in the proofs
of Theorems \ref{cortri} and \ref{corlin}.
\subsection{Chromatic number of rank $k$ hypergraphs}
We will need the following version of the local lemma.

\begin{thmall}[\cite{molloyreed}]\label{asyll}
Consider a set $\mathcal{E} = \{A_1, \dots, A_n\}$ of (typically bad) events that such each $A_i$ is mutually
independent of $\mathcal{E}-(\mathcal{D}_i \cup A_i)$, for some $\mathcal{D}_i \subset \mathcal{E}$. If for each $1 \leq i \leq n$
\begin{itemize}
	\item $\Pr[A_i] \leq 1/4$, and
	\item $\sum_{A_j \in \mathcal{D}_i} \Pr[A_j] \leq 1/4$,
\end{itemize}
then with positive probability, none of the events in $\mathcal{E}$ occur.
\end{thmall}

\begin{lemma}\label{rankkcol}
Fix $k \geq 2$. Let $G$ be a rank $k$ hypergraph with maximum $j$-degree $\Delta_j$,
$j = 2,\dots,k$. Then
\[
\chi(G) \leq \max_{j=2}^k (4k(k-1)\Delta_j)^{1/(j-1)}.
\]
\end{lemma}
\begin{proof}
Set 
\[
r=\max_{j=2}^k (4k(k-1)\Delta_j)^{1/(j-1)}.
\]
Then $\Delta_j \leq \frac{r^{j-1}}{4k(k-1)}$ for each $j = 2, \dots, k$.
Assign each vertex $u \in V(G)$ a color chosen uniformy at random from $[r]$.
For each edge $e \in G$, let $B_e$ denote the event that each vertex in $e$ received
the same color. Then $\Pr[B_e] = r^{1-|e|}$. The event $B_e$ depends only on events $B_f$
such that $B_e \cap B_f \neq \emptyset$.
Since
\[
\sum_{f \in H: e \cap f \neq \emptyset}\Pr[B_f] 
= \sum_{j=2}^k \sum_{\substack{f \in H: e \cap f \neq \emptyset, \\ |f|=j}}\Pr[B_f]
\leq \sum_{j=2}^k |e|\Delta_j r^{1-j}
\leq 1/4,
\]
the Asymmetric Local Lemma implies that there exists a coloring where no event $B_e$ occurs.
\end{proof}

\subsection{Proof of Theorem \ref{cortri}}
We will use the following stronger version of Theorem \ref{trifree3}.
\begin{thm}[\cite{cmcoloring3}]\label{trifreer3}
Suppose $G$ is a rank $3$, triangle-free hypergraph with maximum
$3$-degree $\Delta$ and maximum $2$-degree $\Delta_2$. Then
\[
\chi(G) = O( \max\{ (\frac{\Delta}{\log \Delta})^{1/2}, \frac{\Delta_2}{\log \Delta_2} \}).
\]
\end{thm}

\noindent
In fact, we will prove the following stronger result.
\begin{thm}
Let $G$ be a rank $3$ hypergraph with maximum $3$-degree at most $\Delta$
and maximum $2$-degree at most $\Delta_2$.
Let $\mathcal{H}$ denote the family of rank $3$ triangles.
If 
\[
\Delta_H(G) \leq \max\{\Delta^{1/2}, \Delta_2\}^{v(H)-1}/f
\]
for all $H \in \mathcal{H}$, then
\[
\chi(G) = O(\max \{ \frac{\Delta}{\log f})^{1/2} , \frac{\Delta_2}{\log f}  \}).
\]
\end{thm}

\begin{proof}
By Lemma \ref{rankkcol}, we may assume that $f$ is sufficiently large.
We consider two cases.

\noindent
\textbf{Case 1:} $\Delta_{2} \leq (\Delta \log f)^{1/2}$. 

\noindent
Set $g = \frac{f}{\log^{5/2}f}$.
Since $\Delta_2 \leq (\Delta \log f)^{1/2}$ and $v(H) \leq 6$ for all $H \in \mathcal{H}$,
\begin{equation}\label{dhbound}
\Delta_H(G) 
\leq \max \{\Delta^{1/2}, \Delta_2\}^{v(H)-1}/f
\leq (\Delta \log f)^{\frac{v(H)-1}{2}}/f
\leq \Delta^{\frac{v(H)-1}{2}}/g
\end{equation}
for each $H \in \mathcal{H}$.

Let
\[
K = \{A \subset V(G): |A| = 2 \text{ and } d_3(A) \geq \Delta^{1/2}\}.
\]
Define the rank $3$ hypergraph
\[
G' = G - \{E \in G: A \subset E \text{ for some } A \in K\} \cup K.
\]
In other words, if a pair of vertices has high codegree, then we replace all $3$-edges
containing the pair with a $2$-edge between the pair.
Hence any proper coloring of $G'$ is also
a proper coloring of $G$. 

We will now apply Lemma \ref{partitionlemma} to $G'$ with the function $g' := g^{1/6}$.
Set $\omega_2 = \log^{1/2}g$ and $\omega_3 = 1$.
We first check that $G'$ is $(2\Delta, \omega_2, \omega_3)$-sparse.
Note first that $\omega_2 = g'^{o(1)}$, $\omega_3 = g'^{o(1)}$, and
$G'$ has maximum $3$-degree at most $2\Delta$.
Also, by definition of $K$, 
$G'$ has maximum $(3,2)$-degree at most $\Delta^{1/2} < (2\Delta)^{1/2}\omega_3$.
Let $u \in V(G)$. Since $|L_{G,3}(u)| \leq \Delta$,
$u$ is in at most $2\Delta^{1/2}$ sets in $K$. Thus, 
$G'$ has maximum $(2,1)$-degree at most 
\[
(\Delta \log f)^{1/2} + 2\Delta^{1/2} < (2\Delta\log g)^{1/2} = (2\Delta)^{1/2}\omega_2,
\]
which shows that $G'$ is $(2\Delta, \omega_2, \omega_3)$-sparse.

Let $H' \in \mathcal{H}$. Suppose $T' \subset G'$ is a copy of $H'$
which is not in $G$. Since each $e \in T'-G$ corresponds to at least $\Delta^{1/2}$
edges of size $3$ in $G$, $T'$ corresponds to at least $(\Delta^{1/2}-5)^{|T'-G|}$ copies of $H$
in $G$, where $H$ is the triangle obtained by replacing each $e \in T'-G$
with a distinct size $3$ edge containing $e$ and some vertex outside of $T'$.
Thus 
\[
\Delta_{H'}(G')(\Delta^{1/2}-5)^{v(H)-v(H')} 
 \leq \Delta_H(G).
\]
Using \eqref{dhbound}, this implies 
\[
\Delta_{H'}(G') 
\leq 2\Delta_H(G)\Delta^{(v(H')-v(H))/2}
\leq 2\Delta^{(v(H')-1)/2}/g
\leq 2\Delta^{(v(H')-1)/2}/(g^{1/6})^{v(H')}.
\]

Lemma \ref{partitionlemma} therefore implies
that there exists a partition of $V(G')=V(G)$ into $O(\Delta^{1/2}/g')$ parts such that
each part is $\mathcal{H}$-free and has maximum $j$-degree at most $O(g'^{j-1}\omega_j)$.
By Theorem \ref{trifreer3}, we may properly color each part with
\[
O(\max\{ \frac{g'}{\log^{1/2} g'}, \frac{g' \omega_2}{\log (g' \omega_2)}\})
= O(\frac{g'}{\log^{1/2} g'})
\]
different colors, resulting in a total of 
$O((\frac{\Delta}{\log g'})^{1/2}) = O((\frac{\Delta}{\log f})^{1/2})$ 
colors.

\noindent \textbf{Case 2:} $\Delta_{2} > (\Delta \log f)^{1/2}$.

\noindent
Set $\Delta' = \frac{\Delta_{2}^2}{\log f}$.
Then $\Delta < \Delta'$. Thus $G$ has maximum $3$-degree at most $\Delta'$, and
\[
\Delta_H(G) \leq \max\{\Delta^{1/2}, \Delta_2\}^{v(H)-1}/f < \max\{\Delta'^{1/2}, \Delta_2\}^{v(H)-1}/f
\]
for all $H \in \mathcal{H}$.
We may therefore
apply case 1 with $\Delta'$ in the role of $\Delta$
to obtain a coloring with at most
$
O((\frac{\Delta'}{\log f})^{1/2})
= O(\frac{\Delta_{2}}{\log f})
$
colors.
\end{proof}

\subsection{Proof of Theorem \ref{corlin}}
Let $G$ be a $k$-uniform hypergraph with maximum degree $\Delta$ such that
$\Delta_{k,l}(G) \leq \Delta^{\frac{k-l}{k-1}}/f$ for $l=2,\dots,k-1$.
As in the proof of Theorem \ref{trifreer3}, we may assume that $f$ is sufficiently large.
Set $\omega_j = 1$ for $2 \leq j \leq k$.
Observe that $G$ is $(\Delta, \omega_2, \dots, \omega_k)$-sparse.
For each $l = 2, \dots, k-1$, let
$H_l$ be the $k$-uniform hypergraph consisting of two edges which share exactly $l$ vertices.
Set $\mathcal{H} = \{H_2, \dots, H_{k-1}\}$ and $f' = f^{1/(2k-2)}$.
Then for $a \in V(H_l)$ with $d_{H_l}(a)=2$,
\[
\Delta_{H_l} 
\leq \Delta_{H_l, a} 
\leq \Delta \binom{k-1}{l-1} \Delta_{k,l}
\leq \binom{k-1}{l-1}\Delta^{\frac{2k-1-l}{k-1}}/f
\leq \binom{k-1}{l-1}\Delta^{\frac{2k-1-l}{k-1}}/f'^{v(H_l)},
\]
so Lemma \ref{partitionlemma} implies that there exists a partition
of $V(G)$ into $O(\Delta^{\frac{1}{k-1}}/f')$ parts such that each part is $\mathcal{H}$-free
and has maximum degree at most $O(f'^{k-1})$. Since each part is $\mathcal{H}$-free,
each part is linear and can be colored with $O( \frac{f'}{\log^{1/(k-1)}f'} )$ colors
by Theorem \ref{lineark}.
Using a different set of colors for each part gives a proper coloring
of $G$ with $O((\frac{\Delta}{\log f})^{\frac{1}{k-1}})$ colors.

\section{Concentration results}
\begin{thm}[McDiarmid \cite{mcdiarmid}]\label{mcdiarmid}
For $i=1,\dots,n$, let $\Omega_i$ be a probability space, and let $\Omega = \prod_{i=1}^n\Omega_i$.
For each $i$, let $X_i:\Omega_i \to\ \mathcal{A}_i$ be a random variable,
and let $f:\prod\mathcal{A}_i \to \mathbb{R}$.
Suppose that $f$ satisfies $|f(x)-f(x')| \leq c_i$
whenever the vectors $x$ and $x'$ differ only in the $i^{th}$ coordinate.
Let $Y:\Omega \to \mathbb{R}$ be the random variable $f(X_1,\dots,X_n)$. Then for any $t > 0$,
\[
\Pr[|Y - \Ex[Y]| > t) \leq 2e^{-2t^2 / \sum_{i=1}^n c_i^2}
\]
\end{thm}

\begin{thm}[Kim-Vu \cite{kimvu}]\label{kimvu}
Suppose $F$ is a hypergraph such that $|f| \leq s$ for all $f \in F$. Let
\[
F' = \{f \in F: z_i = 1 \text{ for all } i \in f\},
\]
where the $z_i$, $i \in V(F)$ are independent random variables taking values in $[0, 1]$.
For $A \subset W$ with $|A| \leq s$, let
\[
	Z_A = \sum_{f \in F: f \supset A} \prod_{i \in f-A}z_i.
\]
Let $M_A = \Ex[Z_A]$ and $M_j = \max_{A: |A| \geq j}M_A$ for $j \geq 0$. Then there exist
positive constants $a = a(s)$ and $b = b(s)$ such that for any $\lambda > 0$, 
\[
\Pr\left[| |F'|-\Ex[|F'|]| \geq a\lambda^s\sqrt{M_0 M_1}\right] \leq b|V(F)|^{s-1}e^{-\lambda}.
\]
\end{thm}

\section{Deviation inequality}\label{devinequality}
\noindent
To motivate our next lemma, which is the main novel ingredient in this work,
we outline the proof of Lemma \ref{partitionlemma}, which appears in Section \ref{secpartitionlemma}.
We are given a hypergraph $G$ and a fixed hypergraph $H$, and we know that
each $u \in V(G)$ is in very few copies of $H$.
Our goal is to produce a coloring of $G$ such that the subhypergraph induced by each
color is $H$-free.
We randomly color the vertices of $G$, hoping to remove all
copies of $H$ in the induced subgraph of each color.
Consider some $u \in V(G)$. Since $u$ is in few copies of $H$, 
the expected number of these which remain in $u$'s color is much less than $1$.
We would like to conclude that with very low probability,
$u$ is in more than $c$ copies of $H$ in the subgraph
induced by $u$'s color, for some constant $c$. We could then use the Local Lemma to
find a coloring such that $u$ is in at most $c$ copies of $H$ in the subgraph
induced by $u$'s color, which would bring us close to our goal of removing all copies of $H$ from the subgraph. 
To draw this conclusion, we could try to let $F_u$ be the set of
copies of $H$ which contain $u$, and let $F_u'$ be the set of copies of $H$
which contain $u$ and whose vertices all receive $u$'s color. Then we could apply
Theorem \ref{kimvu} to bound $F_u'$ and make the above conclusion. However, in our case, 
``very low probability'' requires us to set $\lambda = \Omega(\log|V(F_u)|)$.
Since $M_0$ and $M_1$ are always at least $1$ (the empty product is taken to be $1$),
this only allows us to conclude that $u$ is in $\Omega(\log|V(F_u)|)$ copies
of $H$. Since $|V(F_u)|$ will be tending to $\infty$, this is not small enough.

Frieze and the second author~\cite{fmcoloringk} overcame this by using a two-step random process.
However, their proof used the linearity of $G$ and the assumption that $H$ is a triangle,
and we were not able to easily duplicate their method.
Our approach, which can be viewed as a generalization of the method in \cite{aloncoloringsparse}, 
is to instead bound the \emph{transversal number} of the hypergraph $F_u'$.
A \emph{transversal} of $F_u$ is a subset $K \subset V(F_u)$
such that $A \cap K \neq \emptyset$ for all $A \in F_u$. In other words, the 
subhypergraph of $F_u$ induced by $V(F_u)-K$ has no edges. The \emph{transversal number}
of $F_u$, denoted $\tau(F_u)$, is the minimum size of a transversal of $F_u$.

Returning to our outline, we will use Lemma \ref{hittinglemma} below to show that 
the probability that $\tau(F_u')$ is large is very low.

\begin{lemma}\label{hittinglemma}
Suppose $F$ is an $s$-uniform hypergraph,
and $z_i$, $i \in V(F)$ are independent random indicator variables with $\Pr[z_i = 1] = p$, 
for all $i \in V(F)$.
Let 
\[
F' = \{A \in F: \forall i \in A, z_i=1\}.
\]
Suppose there exists $\alpha > 0$ such that $|F|p^{(1-\alpha)s} < 1$.
Then for any $c \geq e2^ss\alpha$,
\[
\Pr[\tau(F') > s^2(c/\alpha)^{s+1}] \leq s^2 |V(F)|^{s-1}p^c.
\]
\end{lemma}

\noindent
We then repeat this for all $v \in V(G)$ in the same color class as $u$.
Let $K_v$ be the set of vertices in a transversal of size at most $\tau(F_v')$.
Following \cite{aloncoloringsparse}, we create a $2$-graph on these vertices,
with an edge from $v$ to each vertex in $K_v$. Since $\tau(F_v')$ is bounded by
a constant, this graph has constant out-degree and can be properly colored
with a constant number of colors. Since the neighborhood of $v$ in this graph
is a transversal and all of the vertices in the neighborhood of $v$ received a different
color than $v$, none of the edges of $F_v$ (which correspond to copies of $H$ containing $v$) 
could survive in $v$'s new color class.
Thus the subgraph induced by each new color contains no copies of $H$.

\subsection{Proof of Lemma \ref{hittinglemma}}
We will need the following simple proposition to prove Lemma \ref{hittinglemma}.
It is a straightforward generalization of the well-known fact that a graph with
many edges has either a large matching or a large star.
Let $F$ be a $k$-uniform hypergraph.
Recall that $M \subset F$ is a \emph{matching} 
if $A \cap B = \emptyset$ for any $A,B \in M$ with $A \neq B$.
When $F$ is $k$-uniform, recall that for all $A \subset V(F)$,
\[
L_F(A) = \{B-A: B \in F \text{ with } A \subset B\}.
\]

\begin{prop}\label{matchingclaim}
If $F$ is a $k$-uniform hypergraph, then
there exists $A \subset V(F)$ such that $L_F(A)$ contains a
matching of size at least $\frac{|F|^{1/k}}{k-|A|}$.
\end{prop}
\begin{proof}
We induct on $k$. If $k = 1$, the claim holds with $A=\emptyset$. Assume the result for $k$,
and let $F$ be a $k+1$-uniform hypergraph with maximum degree $\Delta$.
By the greedy coloring algorithm, $F$ can be partitioned into $k\Delta$ matchings,
so $F$ contains a matching with $\frac{|F|}{k\Delta}$ edges. 
Thus $\Delta > |F|^{k/(k+1)}$ or 
$F$ contains a matching with at least $\frac{|F|}{k |F|^{k/(k+1)}} = \frac{|F|^{1/(k+1)}}{k}$ edges.
In the second case, we are done (with $A=\emptyset$), so 
assume there exists $u \in V(F)$ with $d(u) > |F|^{k/(k+1)}$. 
Set $A = \{u\}$ and consider the $k$-uniform hypergraph $L_F(A)$.
By induction, there exists $B \subset V(L_F(A))$ such that $L_{L(A)}(B)$
contains a matching of size at least 
\[
\frac{|L_F(A)|^{1/k}}{(k-|B|)} 
= \frac{|L_F(A)|^{1/k}}{k+1-|A\cup B|}
> \frac{(|F|^{k/(k+1)})^{1/k}}{k+1-|A\cup B|}
= \frac{|F|^{1/(k+1)}}{k+1-|A\cup B|}.
\]
Since $B \subset V(L_F(A)) \subset V(F)$ and $L_{L(A)}(B) = L_F(A \cup B)$, 
this completes the proof.
\end{proof}

\noindent
\emph{Proof of Lemma \ref{hittinglemma}}.
For $k=0,1,\dots,s$, set $\tau_k = |F|p^{(1-\alpha)k}$. 
For $k=1,\dots,s$, let
\[
H_k = \{ A \subset V(F): |A| =k, d(A) > \tau_k, \text{ and } \forall B \subsetneq A, d(B) \leq \tau_{|B|} \}.
\]

\noindent
Fix $k$, where $1 \leq k \leq s$. Let $B \in \binom{V(F)}{b}$, where $b < k$.
Suppose there exists $A \in H_k$ with $B \subsetneq A$.
By definition of $H_k$, $d(B) \leq \tau_{b}$.
Since each such $A$ corresponds to at least $\tau_k$ sets in $F$, 
and each of these $\tau_k$ sets is counted at most $\binom{s}{k}$ times
under this correspondence, this implies
\begin{equation}\label{heavybound}
\tau_k |\{A \in H_k: B \in \binom{A}{b}\}| \leq \binom{s}{k}d(B) < 2^s \tau_{b}.
\end{equation}
Consider the $k$-uniform hypergraph $Z_k = \{A \in H_k: z_i = 1 \text{ }\forall i \in A\}$.
Suppose $|Z_k| \geq (c/\alpha)^{k}$. 
Then by Proposition \ref{matchingclaim}, there exists $B \in \binom{V(Z_k)}{b}$,
where $b < k$, such that the $(k-b)$-uniform hypergraph $L_{Z_k}(B)$ contains a matching of size 
at least $x_{b} := \frac{c/\alpha}{(k-b)}$.
Each edge in this matching corresponds to an edge $A \in H_k$ with $B \subset A$.
By \eqref{heavybound}, there are at most $2^s \tau_{b}/\tau_k$ edges.
Thus the probability that such a $B$ exists is at most
\begin{align*}
\sum_{b=1}^{k-1}\sum_{B \in \binom{V(H_k)}{b}}\binom{2^s\tau_{b}/\tau_k}{x_{b}}p^{(k-b)x_{b}}
&\leq \sum_{b=1}^{k-1}\sum_{B \in \binom{V(H_k)}{b}}(\frac{e2^s \tau_{b}}{\tau_k x_{b}})^{x_{b}}p^{(k-b)x_{b}} \\
&< \sum_{b=1}^{k-1}\sum_{B \in \binom{V(H_k)}{b}}(\frac{\tau_{b} p^{k-b}}{\tau_k })^{x_{b}} \\
&= \sum_{b=1}^{k-1}\sum_{B \in \binom{V(H_k)}{b}}p^{((1-\alpha)b + k - b - (1-\alpha)k)x_{b}} \\
&= \sum_{b=1}^{k-1}\sum_{B \in \binom{V(H_k)}{b}}p^{\alpha (c/\alpha)} \\
&< k|V(F)|^{k-1} p^{c}. \\
\end{align*}
Therefore, for each $k=1,\dots,s$,
\[
\Pr[|Z_k| > (c/\alpha)^{k}] < k|V(F)|^{k-1}p^c.
\]
Hence
\begin{align*}
\Pr[|\bigcup_{k=1}^s Z_k| \geq s (c/\alpha)^{s}]] 
&< \Pr[|\bigcup_{k=1}^s Z_k| \geq \sum_{k=1}^s (c/\alpha)^{k}] \\
&< \sum_{k=1}^s k |V(F)|^{k-1}p^c \\
&< s^2 |V(F)|^{s-1}p^c.
\end{align*}
Since each edge in $\bigcup_{k=1}^s Z_k$ contains at most $s$ vertices,
this implies
\[
\Pr[|\bigcup_{k=1}^s V(Z_k)| \geq s^2 (c/\alpha)^{s}] < s^2|V(F)|^{s-1}p^c.
\]
We now claim that $\bigcup_{k=1}^s V(Z_k)$ is a transversal of $F'$.
Suppose $A \in F$ and $z_i = 1$ for all $i \in A$.
Since $d(A) = 1 > |F|p^{(1-\alpha)s} = \tau_s$, 
$A \in H_s$ or $d(B) > \tau_{b}$ for some $B \in \binom{A}{b}$.
If $A \in H_s$, then $A \in Z_s$ and $A \subset V(Z_s)$, so assume the second case.
Choose a minimal set $B \in \binom{A}{b}$ with $d(B) > \tau_{b}$. 
Then $B \in H_{b}$, so $B \in Z_{b}$ and hence
$A \cap V(Z_{b}) \neq \emptyset$. \qed

\section{Proof of Lemma \ref{partitionlemma}}\label{secpartitionlemma}
We break the proof of Lemma \ref{partitionlemma} into two steps.
In step 1, we prove Lemma \ref{epslemma}, which is a slight variant of Lemma \ref{partitionlemma}
when $f \geq \Delta^{\epsilon}$.
In Section \ref{secfsmall}, we will use the same argument as in \cite{aloncoloringsparse} to show 
that the proof of Lemma \ref{partitionlemma} can be reduced to Lemma \ref{epslemma}.
\subsection{Step 1: $f \geq \Delta^{\epsilon}$}
\begin{lemma}\label{epslemma}
Fix $k \geq 2$ and $\epsilon \in (0, \frac{1}{k-1})$.
Let $G$ be a rank $k$ hypergraph, and 
let $\mathcal{H}$ be a finite family of fixed hypergraphs.
Suppose that
\begin{itemize}
\item $G$ is $(\Delta, \omega_2, \dots, \omega_k)$-sparse, where 
  $\omega_j = \omega_j(\Delta) = \Delta^{o(1)}$
for all $2 \leq j \leq k$.
\item $\Delta_H(G) \leq \Delta^{\frac{v(H)-1}{k-1}-v(H)\epsilon}$ for all $H \in \mathcal{H}$.
\end{itemize}
Then $V(G)$ can be partitioned into $O( \Delta^{\frac{1}{k-1}-\epsilon})$ parts
such that the hypergraph induced by each part is $\mathcal{H}$-free
and has maximum $j$-degree at most $2\Delta^{(j-1)\epsilon}\omega_j$,
for each $j \in [k]$.

\end{lemma}
\begin{proof}
Let $N = \max_{H \in \mathcal{H}}v(H)$.
Color the vertices of $G$ uniformly at random with $r = \Delta^{\frac{1}{k-1} - \epsilon}$ colors.
Fix $u$, and for each $v \in V(G)$, let $z_{v}$ be a random
indicator variable which is $1$ if $v$ receives the same color as $u$
and $0$ otherwise. Note that $\Pr[z_v = 1] = 1/r$.
For each $H \in \mathcal{H}$, choose $v_H \in V(H)$ such that $\Delta_{H,v_H}(G) = \Delta_H(G)$.
Define a $(v(H)-1)$-uniform hypergraph
\[
T_{H}(u) = \{V(A)-u: A \subset G \text{ with } A \cong_\phi H \text{ and } \phi(u) = v_H \}.
\]
Since to each $A \subset G$ with $A \cong_{\phi} H$ and $\phi(u) = v_H$ we
may associate the set $V(A)-u \in T_H(u)$,
\[
v(T_H(u)) \leq (N-1)|T_H(u)|(N-1) \leq \Delta_H(G).
\]
Also, let
\[
T'_H(u) = \{A \in T_H(u): \forall v \in A, z_v=1\}.
\]
We define $k + |\mathcal{H}|$ bad events for each $u$.
\begin{itemize}
\item
$A_{u,j}$: For $1 \leq j \leq k$, $A_{u,j}$ denotes the event 
\[
Z_{u,j} := \sum_{A \in L_{H, j}(u)} \prod_{i \in A}z_i \geq 2\Delta^{{(j-1)}\epsilon}\omega_j.
\]
\item
$B_{u, H}$: For each $H \in \mathcal{H}$, $B_{u,H}$ denotes the event 
  $\tau(T'_H(u)) > (v(H)-1)^2 (c/\alpha_H)^{v(H)}$, where 
\[
\alpha_H 
= 1-\frac{v(H)\epsilon - \frac{v(H)-1}{k-1}}{(\epsilon-\frac{1}{k-1})(v(H)-1)}
= \frac{\epsilon}{(\frac{1}{k-1} - \epsilon)(v(H)-1)} 
> 0
\]
and 
\[
c = \frac{(N^2 + 3N)2^N}{\frac{1}{k-1}-\epsilon} > 3 N 2^N.
\]

\end{itemize}
To bound the probability of $A_{u,j}$,
we apply Theorem \ref{kimvu} to the $(j-1)$-uniform hypergraph $L_{H,j}(u)$.
Let $A \subset V(L_{H,j}(u))$ with $|A| \leq j-1$. 
Then
\begin{align*}
M_{A} 
= d_j(A\cup \{u\})/r^{j-|A|-1} 
&\leq \Delta_{j,|A|+1}(G)/r^{j-|A|-1} \\
&\leq \Delta^{^\frac{j-|A|-1}{k-1}} \omega_j/r^{j-|A|-1} \\
&= \Delta^{(j-|A|-1)\epsilon} \omega_j.
\end{align*}
Therefore $\Ex[Z_u] = M_0 \leq \Delta^{(j-1)\epsilon}\omega_j$
and $M_1 \leq \Delta^{(j-2)\epsilon}\omega_j$.
Setting $\lambda = (k+3N+2)\log\Delta$, we obtain constants $a$ and $b$ such that
\begin{align*}
\Pr[A_{u,j}] 
= \Pr[Z_{u,j} \geq 2\Delta^{(j-1)\epsilon}\omega_j]
&< \Pr[Z_{u,j} \geq \Delta^{(j-1)\epsilon}\omega_j + a\Delta^{\epsilon(j-1)-\epsilon/2}\lambda^{j-1}] \\
&\leq b(j\Delta)^{j-2}\Delta^{-(j+3N+2)} \\
&< \Delta^{-3N}.
\end{align*}

To bound $B_{u,H}$, we apply Lemma \ref{hittinglemma} with $F = T_H(u)$, $F' = T'_H(u)$, and $p=1/r$.
Since
\[
|T_H(u)|p^{(1-\alpha_H)(v(H)-1)} 
\leq \Delta^{\frac{v(H)-1}{k-1} - v(H)\epsilon}p^{(1-\alpha_H)(v(H)-1)}
= 1,
\]
$\alpha_H \in (0,1)$, and $c > 3N2^N > e 2^N N \alpha_H > e 2^{v(H)-1}(v(H)-1)\alpha_H$,
Lemma \ref{hittinglemma} implies
\begin{align*}
\Pr[B_{u,H}] 
&\leq (v(H)-1)^2 v(T_H(u))^{v(H)-2} p^c \\
&= (v(H)-1)^2 v(T_H(u))^{v(H)-2} \Delta^{(-N^2 - 3N)2^{N}} \\
&< (v(H)-1)^2 v(T_H(u))^{v(H)-2} \Delta^{-N^2 - 3N} \\
&< N^2 (N\Delta^{\frac{N-1}{k-1}})^{N-1} \Delta^{-N^2 - 3N} \\
&< \Delta^{-3N}.
\end{align*}
Each event $A_{u,j}$ is determined by the colors assigned to the vertices in $N_j(u) \cup \{u\}$.
Each event $B_{u,H}$ is determined by the colors assigned to the vertices in $T_H(u) \cup \{u\}$.
Thus $A_{u,j}$ depends on:
\begin{itemize}
\item Events of the form $A_{v,i}$ where $N_j(u) \cap N_i(v) \neq \emptyset$.
    There are at most 
    \[
    (j\Delta^{\frac{j-1}{k-1}}\omega_j)(\sum_{l=1}^k l \Delta^{\frac{l-1}{k-1}}\omega_l) 
     < (j\Delta)(k^2\Delta) < \Delta^{N}
    \]
    such events.
\item Events of the form $B_{v,H}$, where $N_j(u) \cap V(T_H(v)) \neq \emptyset$.
  Since $|N_j(u)| \leq j\Delta$ and $|V(T(v,H))| \leq N\Delta^{\frac{N-1}{k-1}}$,
  there are at most $(j \Delta)|\mathcal{H}|\Delta^{\frac{N-1}{k-1}} \leq \Delta^{N}$ such events.
\end{itemize}
Also, $B_{u,H}$ depends on: 
\begin{itemize}
\item Events of the form $A_{v,j}$ where $V(T_H(u)) \cap N_j(v) \neq \emptyset$.
  There at most $(k^2 \Delta)N\Delta^{\frac{N-1}{k-1}} \leq \Delta^{N}$ such events.
\item Events of the form $B_{v,H'}$, where $V(T_H(u)) \cap V(T_{H'}(v)) \neq \emptyset$.
  There are at most $|\mathcal{H}|N^2\Delta^{2\frac{N-1}{k-1}} \leq \Delta^{2N}$ such events.
\end{itemize}
Since the probability of each event is at most $\Delta^{-3N}$, 
the Local Lemma implies that there exists
a coloring of $V(G)$ with $r$ colors so that none of the events
$A_{u,j}$ or $B_{u,H}$ occur. 

Fix a color, and consider the subhypergraph $G'$ induced by the vertices which received
that color.
For each $u \in V(G')$ and each $H \in \mathcal{H}$, 
let $K(u, H)$ be a minimum sized transversal of $T'(u, H)$.
Create a simple graph $W$ on $V(G')$ with edge set 
$\{\{u,v\}: u \in V(G'), v \in K(u, H) \text{ for some } H \in \mathcal{H}\}$.
Consider any subgraph of $W$ on $n'$ vertices.
Since no event $B_{u,H}$ occurs, the number of edges in this subgraph
is at most $n'|\mathcal{H}|(N-1)^2(c/\alpha)^{N}$; it therefore contains a vertex
with degree at most $2|\mathcal{H}|(N-1)^2(c/\alpha)^{N}$.
$W$ is therefore $2|\mathcal{H}|(N-1)^2(c/\alpha)^{N}$-degenerate and
can thus be properly colored with $2|\mathcal{H}|(N-1)^2(c/\alpha)^{N}+1$ new colors. 
Since each of the $K(u, H)$ is a transversal,
the subhypergraph induced by each of these new colors is $\mathcal{H}$-free.
Repeating this for each of the original $r$ colors results in a partition
of $F$ into at most
\[
r(2|\mathcal{H}|(N-1)^2(c/\alpha)^N + 1) = O(\Delta^{\frac{1}{k-1} - \epsilon})
\]
parts, where each part is $\mathcal{H}$-free and 
has maximum $j$-degree at most $2\Delta^{(j-1)\epsilon}\omega_j$.
\end{proof}

\subsection{Step 2: $f < \Delta^\epsilon$}\label{secfsmall}
When $f < \Delta^{\epsilon}$, we use the same random halving argument as in \cite{aloncoloringsparse}.
We recursively divide the hypergraph into two parts
until we obtain a set of hypergraphs, each with maximum degree small enough to apply
Lemma \ref{epslemma}. 
The halving step is accomplished with Lemma \ref{halflemma}, while
Proposition \ref{seqclaim} is used to analyze the recursion.

\begin{prop}\label{ctlemma}
Let $G$ be a rank $k$, $(\Delta, 2\omega, \dots, 2\omega)$-sparse hypergraph.
Then for any $A \subset V(G)$ and $h \geq |A|$, $G$ contains at most 
$2^{(h+1) T} \omega^T \Delta^{\frac{h-|A|}{k-1}}$
connected hypergraphs on $h$ vertices which contain $A$,
where $T = 2^{h}$.
\end{prop}
\begin{proof}
Starting with any edge that intersects $A$, we try to
greedily grow a connected subgraph which contains $A$ and $h-|A|$
other vertices. At step $i$, we add an edge of size $j_i$
which contains $l_i \geq 1$ vertices already in the subgraph and
$a_i \geq 0$ vertices in $A$ which are not already in the subgraph.
There are at most $2^{h}\Delta_{j_i, l_i+a_i}$ choices for this edge.
Since one edge is added at every step, this process terminates after 
at most $T = 2^{h}$ steps, resulting in a total
of $2^{h T} \prod_{i=1}^{T}\Delta_{j_i, l_i+a_i}$ subgraphs
containing $A$. Since at each step we add $j_i - (l_i+a_i)$ vertices outside of $A$
to the subgraph, $\sum_{i=1}^T j_i-(l_i+a_i) = h-|A|$.
Thus
\[
2^{h T} \prod_{i=1}^{T}\Delta_{j_i, l_i+a_i}
\leq 2^{(h+1) T} \omega^T \prod_{i=1}^{T}\Delta^{\frac{j_i-(l_i+a_i)}{k-1}}
= 2^{(h+1) T} \omega^T \Delta^{\frac{h-|A|}{k-1}}.
\]

\end{proof}
\begin{lemma}\label{halflemma}
Let $G$ be a rank $k$ hypergraph, and 
let $\mathcal{H}$ be a finite family of fixed, connected hypergraphs.
Suppose that
$G$ is $(\Delta, 2\omega_2, \dots, 2\omega_k)$-sparse, where 
$\omega_j = \omega_j(\Delta) = \Delta^{o(1)}$
for all $2 \leq j \leq k$.
Then for $\Delta$ sufficiently large,
there exists a partition of $V(G)$ into two subhypergraphs $G_1$ and $G_2$ such that
\begin{itemize}
\item For each $i=1,2$ and $1 \leq l < j \leq k$, we have
  $\Delta_{j,l}(G_i) \leq \Delta_{j,l}(G)/2^{j-l} + \Delta^{\frac{j-l}{k-1}-\frac{1}{2k}}$ 
\item For each $i=1,2$ and $H \in \mathcal{H}$, we have
  $\Delta_H(G_i) \leq \Delta_H(G)/2^{v(H)-1} + \Delta^{\frac{v(H)-1}{k-1}-\frac{1}{2k}}$.
\end{itemize}

\end{lemma}
\begin{proof}
Let $N = \max_{H \in \mathcal{H}}v(H)$.
Color the vertices of $G$ uniformly at random with the colors $1$ and $2$.
For each $A \subset V(G)$,
let $d'_j(A)$ denote the $j$-degree of $A$ in the subhypergraph
induced by the minimum color of a vertex in $A$. 

For each $H \in \mathcal{H}$, choose $v_H \in V(H)$ such that $\Delta_{H,v_H}(G) = \Delta_H(G)$.
For each $u \in V(G)$, define a $(v(H)-1)$-uniform hypergraph
\[
T_{H}(u) = \{V(A)-u: A \subset G \text{ with } A \cong_\phi H \text{ and } \phi(u) = v_H \}.
\]
Also, for each $H \in \mathcal{H}$, let
\[
T'_H(u) = \{A \in T_H(u): \forall v \in A, v \text{ receives the same color as } u\}.
\]
Define the following bad events:

\begin{itemize}
\item
$C_{A,j}$: For each $A \subset V(G)$ with $d_j(A) > 0$,
  $C_{A,j}$ denotes the event 
\[
d'_j(A) > \Delta_{j,|A|}(G)/2^{j-|A|} + \Delta^{\frac{j-|A|}{k-1} - \frac{1}{2k}}.
\]

\item
$B_{u, H}$: For each $H \in \mathcal{H}$, $B_{u,H}$ denotes the event 
\[
|T'_H(u)| > \Delta_H(G)/2^{v(H)-1} + \Delta^{\frac{v(H)-1}{k-1}-\frac{1}{2k}}.
\]
\end{itemize}

We use Theorem \ref{mcdiarmid} to bound the probability of each event.
The random variable $d'_j(A)$ is determined by the colors of the vertices
in $N_j(A)$. If $v \in N_j(A)$, changing $v$'s color affects $d'_j(A)$ by at most 
$d_j(A \cup \{v\}) \leq \Delta_{j, |A|+1}$.
Also,
\begin{align*}
\sum_{v \in N_j(A)}d_j(A \cup \{v\})^2 \leq \Delta_{j,|A|+1}(G) \sum_{v \in N_j(A)}d_j(A \cup \{v\})
&< \Delta_{j, |A|+1}(G)j d_j(A) \\
&\leq j \Delta_{j, |A|+1}(G) \Delta_{j, |A|}(G) \\
&\leq 4 j \Delta^{\frac{2j-2|A|-1}{k-1}}\omega_j^{2}.
\end{align*}
Consequently, Theorem \ref{mcdiarmid} and $\omega_j(\Delta) = \Delta^{o(1)}$ imply
\begin{align*}
\Pr[C_{A,j}] <
&\Pr[d'_j(A) > \Delta_{j,|A|}(G)/2^{j-|A|} + (4j\Delta^{\frac{2j-2|A|-1}{k-1}} \omega_j^22N \log\Delta)^{1/2}] \\
&\leq 2e^{-4N\log\Delta} \\
&< \Delta^{-3N}.
\end{align*}

Let $T_H(u,v)$ denote the set of copies of $H$ in $G$ containing both $u$ and $v$.
The random variable $|T'_H(u)|$ is determined by the colors of the vertices 
in $V(T_H(u))$. 
Changing the color of $v \in V(T_H(u))$ affects $|T'_H(u)|$ by at most 
$|T_H(u,v)|$, which, by Proposition \ref{ctlemma} (with $|A|=2$), is at most 
$\omega \Delta^{\frac{v(H)-2}{k-1}}$, where 
$\omega = 2^{(v(H)+1)2^{v(H)}} \max_{j=1}^k \omega_j^{2^{v(H)}} = \Delta^{o(1)}$.
Note also that Proposition \ref{ctlemma} (with $|A|=1$) implies $|T_H(u)| < \omega\Delta^{\frac{v(H)-1}{k-1}}$.
Since
\begin{align*}
\sum_{v \in V(T(u,H))} |T_H(u,v)|^2 
&\leq \omega \Delta^{\frac{v(H)-2}{k-1}} \sum_{v \in V(T(u,H))} |T_H(u,v)| \\
&\leq \omega \Delta^{\frac{v(H)-2}{k-1}}v(H)|T_H(u)| \\
&\leq v(H) \Delta^{\frac{v(H)-2}{k-1} + \frac{v(H)-1}{k-1}}\omega^2 \\
&= v(H) \Delta^{\frac{2v(H)-3}{k-1}} \omega^2,
\end{align*}
\begin{align*}
\Pr[B_{u,H}] 
&< \Pr[T'_H(u) > \Delta_H/2^{v(H)-1} + (v(H) \Delta^{\frac{2v(H)-3}{k-1}} \omega^2 2N\log \Delta)^{1/2}] \\
&\leq 2e^{-4 N\log\Delta} \\
&< \Delta^{-3N}.
\end{align*}
The event $C_{A,j}$ is determined by the colors of the vertices in $N_j(A) \cup A$.
The event $B_{u,H}$ is determined by the colors of the vertices in $T_H(u) \cup \{u\}$.
Thus $C_{A,j}$ depends on:
\begin{itemize}
\item Events of the form $C_{B,i}$, where $N_j(A) \cap N_i(B) \neq \emptyset$.
A vertex $u \in N_j(A)$ is in at most $i\Delta 2^i$ sets $N_i(B)$, so there are at most
$(j\Delta) k^2 \Delta 2^k < \Delta^{2N}$ such events.
\item Events of the form $B_{u,H}$, where $N_j(A) \cap V(T_H(u)) \neq \emptyset$.
Since $|V(T_H(u))| \leq N \Delta^{\frac{N-1}{k-1}}$,
there are at most $(j\Delta)|\mathcal{H}|N \Delta^{\frac{N-1}{k-1}} < \Delta^{2N}$ such events.
\end{itemize}
Also, $B_{u,H}$ depends on:
\begin{itemize}
  \item Events of the form $B_{v,H'}$, were $V(T_H(u)) \cap V(T_{H'}(v)) \neq \emptyset$.
    There are at most $|\mathcal{H}|N^2 \Delta^{2\frac{N-1}{k-1}} < \Delta^{2N}$ such events.
  \item Events of the form $C_{A,j}$, were $N_j(A) \cap V(T_H(u)) \neq \emptyset$.
    There are most \\ $(N\Delta^{\frac{N-1}{k-1}})(k^2\Delta 2^k) < \Delta^{2N}$ such events.
\end{itemize}
Since the probability of each event is at most $\Delta^{-3N}$,
the Local Lemma implies that there exists a $2$-coloring of $V(G)$
such that no event $C_{A,j}$ or $B_{u,H}$ holds.
\end{proof}

\begin{prop}\label{seqclaim}
Let $a,b,m \geq 1$ be fixed and $s_0 = d_0^{a/b}g$ for $g > 0$.
Suppose that the sequences $d_t$ and $s_t$ have initial values $d_0$ and $s_0$ and satisfy
\[
d_{t+1} = \frac{d_t}{2^b} + d_t^{1-1/m} \hspace{10pt}\text{and}\hspace{10pt}
s_{t+1} = \frac{s_t}{2^a} + d_t^{a/b - 1/m}.
\]
Then there exists $D > 0$ such that
\[
d_t \leq 2d_02^{-bt} \hspace{10pt}\text{and}\hspace{10pt}
s_t \leq d_t^{a/b}g + d_t^{a/b-1/(2m)}
\]
for all $d_t \geq D$.
\end{prop}
\begin{proof}
Note that for any constant $c$, as $d_0 \to \infty$,
\begin{align*}
(c+d_0^{1/m}2^{-bt/m})^m = d_0 2^{-bt} + \sum_{i=0}^{m-1} \binom{m}{i}c^{m-i}d_0^{i/m}2^{-\frac{bti}{m}} 
= d_02^{-bt} + o(d_0) \leq 2d_02^{-bt}.
\end{align*}
To prove the first inequality, we will thus prove by induction the tighter
bound (for $D$ sufficiently large)
\[
d_t \leq (\frac{2^b}{m(2^{b/m}-1)} + d_0^{1/m}2^{-bt/m})^m.
\]
This is clear for $d_0$, so assume the bound for $d_t$.
Then
\begin{align*}
d_{t+1} 
= \frac{d_t}{2^b} + d_t^{1-1/m}
\leq \frac{1}{2^b}(\frac{2^b}{m}+d_t^{1/m})^m
&\leq \frac{1}{2^b}(\frac{2^b}{m}+\frac{2^b}{m(2^{b/m}-1)}+d_0^{1/m}2^{-bt/m})^m \\
&= (\frac{2^b}{m(2^{b/m}-1)} + d_0^{1/m}2^{-b(t+1)/m})^m.
\end{align*}
For the second inequality, we first prove by induction
$s_{t+1} = \frac{s_0}{x^{t+1}} + \sum_{k=0}^{t}\frac{d_k^{y}}{x^{t-k}}$,
where $x = 2^a$ and $y = a/b - 1/m$.
This is clear for $t = 0$, so assume the bound for $s_t$. 
Then 
\begin{align}\label{sfull}
s_{t+1} = \frac{s_t}{x} + d_t^y = 
\frac{s_0}{x^{t+1}} + \frac{1}{x}\sum_{k=0}^{t-1}\frac{d_k^y}{x^{t-1-k}} + d_t^y =
\frac{s_0}{x^{t+1}} + \sum_{k=0}^{t}\frac{d_k^y}{x^{t-k}},
\end{align}
completing the induction. 
Using $d_k \leq 2d_02^{-bk}$ and $a - by = b/m \geq 0$,
\begin{align}\label{sright}
\sum_{k=0}^t \frac{d_k^y}{x^{t-k}}
\leq \frac{(2d_0)^y}{x^t} \sum_{k=0}^t (\frac{x}{2^{by}})^k
= \frac{(2d_0)^y}{x^t} \sum_{k=0}^t 2^{(a-by) k}
&< \frac{(2d_0)^y}{x^t} \frac{2^{(t+1)(a-by)}}{2^{a-by}-1} \notag \\
&= \frac{(2d_0)^y 2^{(t+1)(a-by)-at}}{2^{b/m}-1}
\end{align}
By definition of $d_t$, $d_t \geq d_02^{-bt}$.
Using this with \eqref{sfull} and with \eqref{sright},
\begin{align*}
s_{t} 
< s_0 2^{-at} + \frac{(2d_0)^y 2^{a-by-tby}}{2^{b/m}-1}
&\leq s_0 (\frac{d_t}{d_0})^{a/b} + \frac{(2 d_t 2^{bt})^y 2^{a-by-tby}}{2^{b/m}-1} \\
&\leq d_0^{a/b}g (\frac{d_t}{d_0})^{a/b} + 
  \frac{d_t^y 2^{y + a - by}}{2^{b/m}-1} \\
&= d_t^{a/b}g + \frac{d_t^{a/b - 1/m} 2^{y+a-by}}{2^{b/m}-1} \\
&< d_t^{a/b}g + d_t^{a/b - 1/(2m)},
\end{align*}
where the last inequality assumes
\[
d_t \geq D > (\frac{2^{y+a-by}}{2^{b/m}-1})^{2m}.
\]
\end{proof}

\noindent
\emph{Proof of Lemma \ref{partitionlemma}}. 
Define sequences 
\begin{align*}
d_{t+1} &= \frac{d_t}{2^{k-1}} + d_t^{1-\frac{1}{2k}} \\
r_{j, l, t+1} &= \frac{r_{j, l, t}}{2^{j-l}} + d_t^{\frac{j-l}{k-1}-\frac{1}{2k}}, 1 \leq l < j \leq k \\
s_{H, t+1} &= \frac{s_{H,t}}{2^{v(H)}-1} + d_t^{\frac{v(H)-1}{k-1} - \frac{1}{2k}}, H \in \mathcal{H},
\end{align*}
where $d_0 = \Delta$, $r_{j,l,0} = \Delta^{\frac{j-l}{k-1}}\omega_j(\Delta)$, 
 and $s_{H, 0} = \Delta^{\frac{v(H)-1}{k-1}}/f^{v(H)}$.
Since $r_{j,l,0} \leq 2\Delta^{\frac{j-l}{k-1}}\omega_j(\Delta)$, 
we may apply Lemma \ref{halflemma} to $G$ to obtain hypergraphs
$G_{1,1}$ and $G_{1,2}$, each with 
maximum degree at most $d_1 = r_{k,1,1}$,
maximum $(j,l)$-degree at most $r_{j,l,1}$, and 
maximum $H$-degree at most $s_{H,1}$.
We apply this halving step a total of $T$ times,
where $T = \lceil \frac{1}{k-1} \log_2 \frac{2 d_0}{f^{4Nk}} \rceil$
and $N = \max_{H \in \mathcal{H}} v(H)$.
Specifically, at each step, we apply Lemma \ref{halflemma} 
(with parameters $d_t = r_{k,1,t},$ $r_{j,l,t},$ $s_{H,t}$) to each of the
hypergraphs $G_{t, 1}, \dots, G_{t, 2^t}$. 
Since $\omega_j(\Delta) = f^{o(1)}$, and for $t \leq T$, 
\begin{equation}\label{dtinfty}
d_t \geq d_0 2^{-(k-1)T} \geq d_0 2^{-(k-1)(1 + \frac{1}{k-1}\log_2\frac{2d_0}{f^{4Nk}})} = 2^{-k} f^{4Nk},
\end{equation}
$\omega_j(\Delta) = d_t^{o(1)}$.
Also, recall that $f$ is sufficiently large, so by \eqref{dtinfty}, we may assume that if $t \leq T$, then
$d_t$ is sufficiently large to apply Proposition \ref{seqclaim}. Thus Proposition \ref{seqclaim}
(with each $r_{j,l,t}$ in the role of $s_t$, $g=\omega_j(\Delta), b=k-1, a=j-l$, and $m=2k$)
implies
\[
r_{j,l,t} \leq d_t^{\frac{j-l}{k-1}}\omega_j(\Delta) + d_t^{\frac{j-l}{k-1} - \frac{1}{4k}} 
\leq 2d_t^{\frac{j-l}{k-1}}\omega_j(\Delta).
\]
Thus each $G_{t,i}$ is $(d_t, 2\omega_2, \dots, 2\omega_k)$-sparse, 
so we may apply Lemma \ref{halflemma} to obtain $2^{t+1}$ new hypergraphs
$G_{t+1,1}, \dots, G_{t+1,2^{t+1}}$ such that each $G_{t+1,i}$ has
maximum degree $d_{t+1} = r_{k,1,t+1}$,
maximum $(j,l)$-degree $r_{j, l, t+1}$, and maximum $H$-degree $s_{H, t+1}$.
In the final step, we obtain $2^T$ hypergraphs $G_{T, 1}, \dots, G_{T, 2^T}$,
each with maximum degree $d_T$, maximum $(j,l)$-degree $r_{j, l, T}$ and maximum $H$-degree $s_{H, T}$.
By Proposition \ref{seqclaim}, 
\[
d_T \leq 2 d_0 2^{-(k-1)T} \leq f^{4Nk} < 2^k f^{4Nk},
\]
and
\[
r_{j,l,T} \leq 2d_T^{\frac{j-l}{k-1}}\omega_j(\Delta) 
\leq ( 2^k f^{4Nk})^{\frac{j-l}{k-1}}\omega_j(\Delta).
\]
Proposition \ref{seqclaim} (with each $s_{H,T}$ in the role of $s_T$, 
$g=1/f^{v(H)}, b=k-1, a=v(H)-1,$ and $m=2k$) also yields
\begin{align*}
s_{H,T} 
&\leq d_T^{\frac{v(H)-1}{k-1}} / f^{v(H)} + d_T^{\frac{v(H)-1}{k-1}-\frac{1}{4k}} \\
&\leq (f^{4Nk})^{\frac{v(H)-1}{k-1}} / f^{v(H)} + (f^{4Nk})^{\frac{v(H)-1}{k-1}-\frac{1}{4k}}  \\
&\leq  2(f^{4Nk})^{\frac{v(H)-1}{k-1}} / f^{v(H)} \\
&= 2(f^{4Nk})^{\frac{v(H)-1}{k-1} - \frac{v(H)}{4Nk}} \\
&< (2^k f^{4Nk})^{\frac{v(H)-1}{k-1} - \frac{v(H)}{4Nk}}.
\end{align*}
We may therefore apply Lemma \ref{epslemma}
(with $\Delta = 2^k f^{4Nk}, \omega_j(\Delta),$ 
and $\epsilon = \frac{1}{4Nk}$)
to partition each of the hypergraphs $G_{T,1}, \dots, G_{T, 2^T}$
into $O(f^{4Nk(\frac{1}{k-1} - \frac{1}{4Nk})})$ parts such that the hypergraph
induced by each part is $\mathcal{H}$-free and has maximum $j$-degree at most
$2 (2^k f^{4Nk})^{\frac{j-1}{4Nk}}\omega_j < 2^{2k}f^{j-1}\omega_j$ for each $1 \leq j \leq k$.
Summing over each of the $2^T$ hypergraphs, this results in a total
of
\[
2^T O(f^{\frac{4Nk}{k-1}-1})
\leq 2^{1+\frac{1}{k-1}\log_2\frac{2d_0}{f^{4Nk}}} O(f^{\frac{4Nk}{k-1}-1})
= O(\Delta^{\frac{1}{k-1}}/f)
\]
$\mathcal{H}$-free hypergraphs, each with maximum $j$-degree at most $2^{2k}f^{j-1}\omega_j(\Delta)$.
\bibliographystyle{amsplain}
\bibliography{bib}
\end{document}